\documentclass[a4paper,12pt]{amsart}

\DeclareRobustCommand{\SkipTocEntry}[5]{}

\usepackage{ textcomp }

\usepackage{enumerate}
\usepackage{shadow}
\usepackage{wasysym}

\usepackage{accents}

\usepackage{marvosym}

\usepackage[T1]{fontenc}
\usepackage{empheq}
\usepackage{relsize}
\usepackage{amsmath}

\usepackage{bm}

\usepackage{upgreek}
\usepackage{ esint }
\usepackage{color}
\usepackage{comment}
\usepackage{amssymb}
\usepackage{amsfonts}
\usepackage{graphicx}

\usepackage{amscd}

\usepackage{slashed}
\usepackage{pdflscape}
\usepackage{tikz}
\usepackage{enumerate}
\usepackage{multirow}
\usepackage{stmaryrd}
\usepackage{cancel}

\usepackage{scalerel,stackengine}

\usepackage{ifthen}

\usepackage{bbold}
\usepackage[linkcolor=blue]{hyperref}
\usepackage{mathrsfs}
\usepackage[colorinlistoftodos]{todonotes}

\definecolor{blue}{rgb}{.255,.41,.884} 
\definecolor{red}{rgb}{1, 0, 0} 
\definecolor{green}{rgb}{.196,.804,.196} 
\definecolor{yellow}{rgb}{1,.648,0} 
\definecolor{pink}{rgb}{1,0.5,0.5}

\newtheorem{theorem}{Theorem}[section]
\newtheorem{lemma}[theorem]{Lemma}

\newtheorem{corollary}[theorem]{Corollary}

\theoremstyle{definition}
\newtheorem{definition}[theorem]{Definition}
\newtheorem{example}[theorem]{Example}

\theoremstyle{remark}
\newtheorem{remark}[theorem]{Remark}

\usepackage{pdfsync}

\usepackage{graphicx} 
\usepackage{amsmath} 
\usepackage{amsfonts}
\usepackage{amssymb}

\newcommand{\be}{\begin{equation}}

\newcommand{\ee}{\end{equation}}

\newcommand{\om}{\omega}

\newcommand{\si}{\sigma}

\newcommand{\ba}{\begin{array}}

\newcommand{\ea}{\end{array}}

\newcommand{\beq}{\begin{eqnarray}}

\newcommand{\eeq}{\end{eqnarray}}

\newtheorem{lm}{lemma}

\newtheorem{thee}{theorem}

\newtheorem{proo}{proposition}

\newtheorem{co}{corollary}

\newtheorem{rem}{remark}

\newtheorem{deff}{definition}

\newcommand{\bd}{\begin{deff}}

\newcommand{\ed}{\end{deff}}

\newcommand{\bl}{\begin{lm}}

\newcommand{\el}{\end{lm}}

\newcommand{\bp}{\begin{proo}}

\newcommand{\ep}{\end{proo}}

\newcommand{\bt}{\begin{thee}}

\newcommand{\et}{\end{thee}}

\newcommand{\bc}{\begin{co}}

\newcommand{\ec}{\end{co}}

\newcommand{\brm}{\begin{rem}}

\newcommand{\erm}{\end{rem}}

\usepackage{t1enc}

\def\Cal{\mathcal}

\newcommand{\bS}{\mathbb{S}}

\newcommand{\newc}{\newcommand}

\let\ccdot.

\newc{\aR}{\mbox{\boldmath{$ R$}}}
\newc{\aS}{\mbox{\boldmath{$ S$}}}
\newc{\aT}{\mbox{\boldmath{$ T$}}}
\newc{\aW}{\mbox{\boldmath{$ W$}}}

\newc{\aD}{\mbox{\boldmath{$ D$}}\hspace{-.2mm}}

\newc{\aK}{\mbox{\boldmath{$ K$}}}
\newc{\aL}{\mbox{\boldmath{$ L$}}}

\newcommand{\ce}{{\Cal E}}

\newcommand{\ct}{{\Cal T}}

\newcommand{\bT}{{\Bbb T}}

\usepackage{amssymb}
\usepackage{amscd}

\newcommand{\Up}{\Upsilon}

\newcommand{\cT}{{\mathcal T}}

\let\hash=\sharp  

\newcommand{\cV}{{\Cal V}}

\newcommand{\nn}[1]{(\ref{#1})}

\newcommand{\bg}{{\sf g}}

\newc{\obstrn}[2]{B^{#1}_{#2}}

\newcommand{\rpl}                         
{\mbox{$
\begin{picture}(12.7,8)(-.5,-1)
\put(0,0.2){$+$}
\put(4.2,2.8){\oval(8,8)[r]}
\end{picture}$}}

\newcommand{\lpl}                         
{\mbox{$
\begin{picture}(12.7,8)(-.5,-1)
\put(2,0.2){$+$}
\put(6.2,2.8){\oval(8,8)[l]}
\end{picture}$}}

\newc{\tensor}[1]{#1}
\newc{\Mvariable}[1]{\mbox{#1}}
\newc{\down}[1]{{}_{#1}}
\newc{\up}[1]{{}^{#1}}

\newc{\JulyStrut}{\rule{0mm}{6mm}}
\newc{\midtenPan}{\mbox{\sf S}}
\newc{\midten}{\mbox{\sf T}}
\newc{\midtenEi}{\mbox{\sf U}}
\newc{\ATen}{\mbox{\sf E}}
\newc{\BTen}{\mbox{\sf F}}
\newc{\CTen}{\mbox{\sf G}}

\def\sideremark#1{\ifvmode\leavevmode\fi\vadjust{\vbox to0pt{\vss
 \hbox to 0pt{\hskip\hsize\hskip1em
 \vbox{\hsize2cm\tiny\raggedright\pretolerance10000
  \noindent #1\hfill}\hss}\vbox to8pt{\vfil}\vss}}}

\numberwithin{equation}{section}

\newcommand{\hh}{{\hspace{.3mm}}}

\newcommand{\cc}{\boldsymbol{c}}

\newcommand{\pdot}{{\boldsymbol{\cdot}}}

\newcommand{\sss}{\scriptscriptstyle}

\renewcommand\geq{\geqslant}
\renewcommand\leq{\leqslant}

\DeclareMathOperator{\EXT}{d}
\newcommand{\ext}{{\EXT\hspace{.01mm}}}

\newcommand{\Dslash}{\cancel{D}}

\stackMath
\newcommand\reallywidehat[1]{%
\savestack{\tmpbox}{\stretchto{%
  \scaleto{%
    \scalerel*[\widthof{\ensuremath{#1}}]{\kern-.6pt\bigwedge\kern-.6pt}%
    {\rule[-\textheight/2]{1ex}{\textheight}}
  }{\textheight}%
}{0.5ex}}%
\stackon[1pt]{#1}{\tmpbox}%
}

\begin{document}


\renewcommand{\today}{}
\title{
{Einstein \& Yang--Mills\\[1mm] implies Conformal Yang--Mills
\\
\phantom{X}
}}

\medskip

\author{ Samuel Blitz${}^\diamondsuit$, A. Rod Gover${}^{\heartsuit}$, Jaros\l aw Kopi\'nski${}^{\clubsuit}$, \&  Andrew Waldron${}^\spadesuit$}

\address{${}^\diamondsuit$
 Department of Mathematics and Statistics \\
 Masaryk University\\
 Building 08, Kotl\'a\v{r}sk\'a 2 \\
 Brno, CZ 61137} 
   \email{blitz@math.muni.cz}
   
   \address{${}^\heartsuit$
  Department of Mathematics\\
  The University of Auckland\\
  Private Bag 92019\\
  Auckland 1142\\
  New Zealand} \email{r.gover@auckland.ac.nz}

  \address{${}^{\clubsuit}$
  Center for Quantum Mathematics and Physics (QMAP)\\
  Department of Mathematics\\ 
  University of California\\
  Davis, CA95616, USA} \email{jkopinski@ucdavis.edu}

  \address{${}^{\spadesuit}$
  Center for Quantum Mathematics and Physics (QMAP)\\
  Department of Mathematics\\ 
  University of California\\
  Davis, CA95616, USA} \email{wally@math.ucdavis.edu}

\vspace{10pt}

\renewcommand{\arraystretch}{1}

\begin{abstract}
There exist  conformally invariant, higher-derivative, variational an\-alogs of the Yang--Mills condition for  connections on vector bundles over a conformal manifold of even dimension greater than or equal to six. We give a compact formula for these analogs  
and prove that they are a strict weakening of the  Yang--Mills condition with respect to an Einstein metric.
We also show that 
the
conformal Yang--Mills condition for the tractor connection of an even dimensional  conformal manifold is equivalent to vanishing of its Fefferman--Graham obstruction tensor. This result uses that the  tractor connection on a Poincar\'e--Einstein manifold is itself Yang--Mills.

\vspace{.8cm}

\noindent
\begin{center}
\end{center}

\vspace{-0.1cm}

\end{abstract}


\maketitle

\pagestyle{myheadings} \markboth{Blitz, Gover, Kopi\'nski, \& Waldron}{Einstein and Yang--Mills}



\newcommand{\balpha}{{\bm \alpha}}
\newcommand{\balphas}{{\scalebox{.76}{${\bm \alpha}$}}}
\newcommand{\bnu}{{\bm \nu}}
\newcommand{\blambda}{{\bm \lambda}}
\newcommand{\bnus}{{\scalebox{.76}{${\bm \nu}$}}}
\newcommand{\bnuss}{\hh\hh\!{\scalebox{.56}{${\bm \nu}$}}}

\newcommand{\bmu}{{\bm \mu}}
\newcommand{\bmus}{{\scalebox{.76}{${\bm \mu}$}}}
\newcommand{\bmuss}{\hh\hh\!{\scalebox{.56}{${\bm \mu}$}}}

\newcommand{\btau}{{\bm \tau}}
\newcommand{\btaus}{{\scalebox{.76}{${\bm \tau}$}}}
\newcommand{\btauss}{\hh\hh\!{\scalebox{.56}{${\bm \tau}$}}}

\newcommand{\bsigma}{{\bm \sigma}}
\newcommand{\bsigmas}{{{\scalebox{.8}{${\bm \sigma}$}}}}
\newcommand{\bbeta}{{\bm \beta}}
\newcommand{\bbetas}{{\scalebox{.65}{${\bm \beta}$}}}

\renewcommand{\bS}{{\bm {\mathcal S}}}
\newcommand{\bB}{{\bm {\mathcal B}}}
\renewcommand{\bT}{{\bm {\mathcal T}}}
\newcommand{\bM}{{\bm {\mathcal M}}}

\newcommand{\go}{{\mathring{g}}}
\newcommand{\nuo}{{\mathring{\nu}}}
\newcommand{\alphao}{{\mathring{\alpha}}}

\newcommand{\Ell}{\mathscr{L}}
\newcommand{\density}[1]{[g\, ;\, #1]}

\renewcommand{\Dot}{{\scalebox{2}{$\cdot$}}}

\newcommand{\PanE}{P_{4}^{\sss\Sigma\hookrightarrow M}}
\newcommand\eqSig{ \mathrel{\overset{\makebox[0pt]{\mbox{\normalfont\tiny\sffamily~$\Sigma$}}}{=}} }
\renewcommand\eqSig{\mathrel{\stackrel{\Sigma\hh}{=}} }
\newcommand\eqtau{\mathrel{\overset{\makebox[0pt]{\mbox{\normalfont\tiny\sffamily~$\tau$}}}{=}}}
\newcommand{\hd }{\hat{D}}
\newcommand{\hdb}{\hat{\bar{D}}}
\newcommand{\Two}{{{{\bf\rm I\hspace{-.2mm} I}}{\hspace{.2mm}}}{}}
\newcommand{\TwoN}{{\mathring{{\bf\rm I\hspace{-.2mm} I}}{\hspace{.2mm}}}{}}
\newcommand{\Fn}{\mathring{\mathcal{F}}}
\newcommand{\csdot}{\hspace{-0.75mm} \cdot \hspace{-0.75mm}}
\newcommand{\IdD}{(I \csdot \hd)}
\newcommand{\Kd}{\dot{K}}
\newcommand{\Kdd}{\ddot{K}}
\newcommand{\Kddd}{\dddot{K}}

 \newcommand{\bdot }{\mathop{\lower0.33ex\hbox{\LARGE$\cdot$}}}

\definecolor{ao}{rgb}{0.0,0.0,1.0}
\definecolor{forest}{rgb}{0.0,0.3,0.0}
\definecolor{red}{rgb}{0.8, 0.0, 0.0}

\newcommand{\APE}[1]{{\rm APE}_{#1}}
\newcommand{\PE}{{\rm PE}}
\newcommand{\FF}[1]{\mathring{\underline{\overline{\rm{\scalebox{.8}{$#1$}}}}}}
\newcommand{\ltots}[1]{{\rm ltots}_{#1}}

\newcommand{\FFdn}{\FF{\hh \hh d\hh\hh }^{\sss\rm DN}}

\newcommand{\FFdnf}{\FF{\, d\, }^{\sss\rm \hh DN^\flat}}

 \newcommand{\Trace}{\operatorname{Tr}}
 \newcommand{\cK}{\mathcal{K}}


\section{Introduction}

The  problem of finding distinguished linear connections $A$ on a vector bundle $$V\to VM\to M$$
has myriad applications
ranging from fundamental physics~\cite{YandM,BPST} to four-manifold theory~\cite{Don1,Don2}. When the base manifold $M$ is  equipped with a (possibly indefinite  signature) metric $g$, a natural criterion on $A$ is the {\it Yang--Mills condition}~\cite{YandM}
$$
j_a[g,A]:=g^{bc}\nabla^{A,g}_b F^A_{ca} =0\, ,
$$
where $F^A$ is the curvature of $A$ and $\nabla^{A,g}$ is the corresponding Levi--Civita coupled covariant derivative.
When  $M$ is a four-manifold, any $A$ whose {\it Yang--Mills current}~$j[g,A]$  vanishes also solves 
$$
j[\Omega^2g,A]=0\, ,
\:
\mbox{ where $0<\Omega\in C^\infty(M)$.}
$$
In other words, the four dimensional Yang--Mills condition  is conformally invariant.
Consider now a space of connections on $VM$ such that the 
natural pairing on endomorphisms  of $\Gamma(VM)$ induces a non-degenerate pairing 
on the corresponding subspace of endormorphisms inhabited by curvatures $F$, 
or in short ``the space of connections has a non-degenerate trace-form''.
Then
the current
$j$ is the functional gradient
 of an energy  
$$
E_{\rm YM}\big[g,A] := -\frac14 \int_M  \ext {\rm Vol}(g)\Trace F_{ab} F^{ab}\, .
$$
When $M$ is a four-manifold, $E[\Omega^2 g,A]=E[g,A]$ so the energy functional depends only on the conformal class of metrics $c:=[g]$.
In dimensions other than four, this is no longer the case.

Remarkably, there exists a variational,  conformally invariant analog of the Yang--Mills condition for six-manifolds~\cite{BGforms,GPS,GLWZ1,GLWZ2}, namely
\begin{equation}\label{k}
k_a[g,A]:=\tfrac12
 \nabla^c
\big(
  \nabla_{[c} {j}_{a]} 
 - 4 P_{[c}{}^{b} F_{a]b}
- J  F_{ca}
\big)
+ \tfrac12 [j^b, F_{ba}] =0\, .
\end{equation}
The above current $k$ is the functional gradient of 
$$
E_{\rm CYM}\big[[g],A\big] := \frac14
\int_M
 \ext {\rm Vol}(g)\Trace 
 \big(j_a j^a +J F_{ab}F^{ab} +4 {F}^{ab} P_{a}{}^c F_{bc} - \nabla^a v_a\big)\, ,
$$
where $v_a:=\frac1 4  \nabla_a F^2+2F_{ab} j^b$. Moreover 
$$
k[\Omega^2 g,A] = \Omega^{-4} k[g,A]\, ,
$$
so we term $k[g,A]=0$ the (dimension six) {\it conformal Yang--Mills condition}. Our conventions for tensors on manifolds
are detailed in Section~\ref{Sleep-Well-Arena}.

\smallskip

It not difficult to find an exact 2-form $F=\ext A$ in Euclidean space $({\mathbb R}^6,\delta)$ that is not coclosed but such that~$k[\delta,A]=0$. 
{\it Viz.}, there exist conformal Yang--Mills connections  that are themselves not Yang--Mills. Conversely, 
 if $A_{\rm YM}$ is a Yang--Mills connection, it need not be that $k[g,A_{\rm YM}]=0$. This is almost evident from Equation~\nn{k}, but we give explicit examples in Section~\ref{examples}.
However, if additionally the metric $g$ is an {\it Einstein metric} $g_{\rm E}$, meaning that its trace-free Schouten tensor $P$ vanishes, then one immediately sees from Equation~\nn{k} that
$$
k[g_{\rm E},A_{\rm YM}]=0\, ,
$$
so Einstein and Yang--Mills implies conformal Yang--Mills for six-manifolds.

\medskip

\noindent
This suggests the following two questions:
\begin{enumerate}[(i)]
\item Do there exist variational, higher derivative,  conformal Yang--Mills conditions in  higher dimensions?
\item\label{q2} If so, are any of these 
weakenings of the Einstein and Yang--Mills condition?
\end{enumerate}
The first question has been answered in the affirmative for all even dimensions greater than or equal to four.
\begin{theorem}[from~\cite{GLWZ1}]\label{big_formaggio}
Let $n\in 2{\mathbb Z}_{\geq 2}$ and let~$(M^n,c)$ be a conformal  $n$-manifold. In addition let $VM$ be a vector bundle  with a connection~$A$.
 Then  there exists a 1-form~$k$ on~$M$ valued in endomorphisms of $VM$ such that, for any $g\in c$ and $0<\Omega\in  C^\infty(M)$,
\begin{equation}\label{Iamadensity}
k[\Omega^2 g,A] = \Omega^{2-n} k[g,A]\, .
\end{equation}
If in addition the  space of connections is equipped with a non-degenerate trace-form, then
 $k$ is the functional 
gradient of a 
  corresponding energy 
$$
E_{\rm CYM}\big[[g],A\big] := -\frac14
\int_M
 \ext {\rm Vol}(g)
 \mathcal{E}\big[g,A]\, ,
 $$
 where
 $$
  \mathcal{E}[\Omega^2 g,A] = \Omega^{-n}  \mathcal{E}[g,A]\, .
 $$
\end{theorem}
\noindent
Equation~\nn{Iamadensity} implies that we may view $k$ as a weight $2-n$ conformal density determined only by the conformal class itself; when doing so we  then write~$k[c,A]$---see Section~\ref{Tiramisu} for details.

The proof of Theorem~\ref{big_formaggio}
 in~\cite{GLWZ1} relies on ``geometric holography'', where one embeds a conformal manifold as a hypersurface boundary of some $n+1$ dimensional manifold,  
$$
M^{n}\hookrightarrow X^{n+1}\, .
$$
The data of a conformal class of metrics $c$ on $M$   defines a complete metric $g_+$ on the interior $X_+$ of $X$ subject to the asymptotic Einstein condition $$
P^{g_+}+\frac12 \hh g_+ = {\mathcal O}(s^{n-2})\, ,
$$
where $s$ is any defining function for the embedding $M\hookrightarrow X$ and $s^2 g_+$ smoothly extends as a metric to $X$. Moreover $g_+$ is determined uniquely modulo terms of sufficiently high order in $s$~\cite{FG}. In turn, given a connection $A$ on $VM$, this can be extended to a Yang--Mills connection  ${\bm A}$ over~$X_+$ with respect to $g_+$, at least up to a certain order, namely
$$
j[g_+,{\bm A}]={\mathcal O}(s^{n-2})\, .
$$
In fact, 
$$
s^{2-n} j[g_+,{\bm A}]\big|_{M}= k[g,A]\, ,
$$
where the metric $g\in c$ is the  smooth extension of $s^2 g_+$ to the boundary $M$~\cite{GLWZ1}. The proof of Theorem~\ref{big_formaggio} establishes that the obstruction $k$  to smooth Yang--Mills solutions on $X$ defines a conformally invariant and variational  1-form on~$M$ valued in endomorphisms of $VM$. Here we shall call the so-obtained tensor $k$  the {\it conformal Yang--Mills current} of the connection $A$.

While the work of~\cite{GLWZ1} gave both a  compact holographic formula for the energy functional $
E_{\rm CYM}[g,A]$ as well as a recursion for computing explicitly the corresponding higher conformal Yang--Mills currents $k$,  higher dimensional analogs of Equation~\nn{k}  quickly become unwieldy.
However, tractor calculus methods  allow for a simple universal formula.

\begin{theorem}\label{parmesan}
Let $n\in 2{\mathbb Z}_{\geq 3}$ and  $(M^n,c)$ be a conformal  manifold. If $A$  is a connection on $VM$ with curvature tractor 
  ${\mathcal F}= q_{\sf w}(F)$ and
  conformal Yang--Mills current $k$, then 
\begin{equation}\label{P=k}
\cancel{P}_{n-4} {\mathcal F}=2(n-3)(n-4)[(n-5)!]^2\, 
 q_{\sf s}(k)
\, .
\end{equation}
\end{theorem}
\noindent
The tractor bundle $\ct M$ for an $n$-dimensional conformal manifold $(M,c)$ is a rank~$n+2$ generalization of the tangent bundle $TM$ that is well-adapted to the study of conformal geometry; see Section~\ref{Tiramisu}. In the above, $\mathcal F=q_{\sf w}(F)$ is a tractor extension of the curvature~$F$ of a connection  $A$, and $q_{\sf s}(k)$ is a tractor algebraically in isomorphism with the conformal Yang--Mills current $k$. The operator  $\cancel{P}_{n-4}$ is a derivative-order $n-4$, connection-coupled analog of the conformally invariant Laplacian powers of Graham, Jennes, Mason  and Sparling (GJMS)~\cite{GJMS}.

An important invariant of even dimensional conformal manifolds of dimension greater than  or equal to six is the Fefferman--Graham obstruction tensor ${\mathcal B}$~\cite{FG}, which may be viewed as a higher dimensional analog of the Bach tensor; see Section~\ref{Gnocchi} for further details. Moreover, the tractor bundle  $\ct M$ itself comes canonically equipped with a connection~$A^\ct$ called the tractor connection.
Our next result relies on holographic methods---see  Theorem~\ref{YMFG}---to show that
   the Fefferman--Graham obstruction tensor  is equivalent to the conformal Yang--Mills current  of the tractor connection. This extends what was known in dimensions~$n=4,6$~\cite{Kaku,GSS,GPS}.\begin{theorem}\label{CYM=FG}
Let $(M^n,c)$ be a conformal manifold of dimension $n\in 2{\mathbb Z}_{\geq 2}$.~Then the conformal Yang--Mills current $k$ for the tractor connection and the Fefferman--Graham obstruction tensor~${\mathcal B}$
are related by a bundle monomorphism given by the restriction to $\Gamma(\odot^2 T^*M[2-n])$ of $$q_{\sf s}:\Gamma(T^*M[2-n]\otimes T^*M) \longrightarrow  \Gamma( \wedge^2\ct M[2-n]\otimes T^*M)\, ,$$ 
and such that
\begin{equation}\label{propto}
k[c,A^\ct] =2 (n-1)(n-2)  q_{\sf s}({\mathcal B})\, .
 \end{equation}
 \end{theorem}

Our next result answers Question~(\ref{q2}) in the affirmative.

\begin{theorem}\label{main}
Let $n\in 2{\mathbb Z}_{\geq 2}$ and let $(M^n,c)$ be a conformal  manifold. In addition let   $VM$ be a  vector bundle over $M$. 
 Then the  Einstein and Yang--Mills conditions imply the conformal Yang--Mills condition. In other words,
 if $g_{\rm E}\in c$ is an Einstein metric and $A_{\rm YM}$
 is a Yang--Mills connection with respect to $g_{\rm E}$, 
 then
 $$
k[g_{\rm E},A_{\rm YM}]=0\, .
$$
\end{theorem}

 Prescient elements of the tractor calculus
 and holography, on which Theorems~\ref{parmesan} and~\ref{main} rely, 
 are presented (and
 slightly extended to handle our Yang--Mills setting)
  in Sections~\ref{Tiramisu}
and~\ref{Gnocchi}, respectively.

\subsection{Convention center}\label{Sleep-Well-Arena}

We employ the abstract index
notation of Penrose~\cite{Penrose} to handle tensor bundles and sections thereof.
For that we denote by $\ce M$, or simply~$\ce$, the trivial line bundle over $M$  with smooth section space $C^\infty(M)$. 
Tensor bundles are then described by attaching appropriate indices to~$\ce$. 
  Round and square brackets are then used to denote totally symmetric
  and antisymmetric subbundles. 
Thus  $\ce_a$ and $\ce^a$ are  respective notations for the
cotangent and tangent bundles~$T^*M$ and~$TM$ while $v^a$ stands for a section of the latter and~$\ce_{(ab)}$ denotes
the subbundle~$\odot^2T^*M$ of symmetric tensors in $T^*M\otimes T^*M$.
Given a metric~$g$~(say), then $g_{ab} \cdot \ce$ is a line-subbundle of~$\ce_{(ab)}$.
The contraction~$\om (v) =:\iota_v \omega$ of a
1-form~$\om$ with a tangent vector $v^a$ is denoted by a repeated index~$v^a\om_a$.  
We will in addition use the notation $\omega_v$ for this, and along the same lines, symbols such as $R_{avcd}$ denote~$v^b R_{abcd}$.
The interior multiplication operator $\iota$  also makes sense for $v^\alpha \in \Gamma(VM)$ and $\omega_{\beta_1\cdots \beta_k}\in \Gamma(\wedge^kV^*M)$, where $VM$ is any vector bundle, so that
$$
\iota_v \omega := v^\alpha \omega_{\alpha \beta_2 \cdots \beta_k}\, .
$$
The above notation will also be applied to the case when $v$ is some linear operator taking values in $\Gamma(VM)$.

If $X^\alpha{}_\beta$ is an endomorphism of $\Gamma(VM)$, we may denote the action of~$X$ by 
$$X^\hash v^\alpha:=X^\alpha{}_\beta v^\beta\:
\mbox{ and } 
\:
X^\hash \mu_\beta:= -X^\alpha{}_\beta \mu_\alpha\, ,
$$ 
where $\mu_\alpha\in \Gamma(V^*M)$; the {\it hash action} on higher tensor products is defined  by the Leibniz rule, while the hashed action $X^\hash$ annihilates scalars. 
If $A$ is a linear connection on $VM$ and $v^\alpha$ a section thereof, then the curvature $F^A$ is defined in terms of a commutator by coupling $\nabla^A$ to any torsion-free affine connection~$\gamma$ and writing
$$
[\nabla_a^{A,\gamma},\nabla^{A,\gamma}_b] v^\alpha = 2 \nabla_{[a}\nabla_{b]} v^\alpha =  F^A_{ab}{}^\hash v^\alpha=F_{ab}{}^\alpha{}_\beta v^\beta\, .
$$
Throughout, we will suppress identifying sub and superscripts where context allows and simply write $\nabla$ and $F$ for the respective connection-coupled gradient and curvature.
We may of course write $X^\hash =: X^\alpha{}_\beta\,  \hash^\beta{}_{\alpha}$ in order to define an operator~$ \hash^\beta{}_{\alpha}$. In this way more complicated operators such as $X^{\hash\hash}:=X^\alpha{}_\beta{}^\gamma{}_\delta \,\hash^\beta{}_{\alpha} \hash^\delta{}_{\gamma}$
may be introduced.

Our convention for the
Riemann tensor $R_{ab}{}^c{}_d$
of the Levi--Civita connection $\nabla_a^g$ of a metric $g_{ab}$,
 is such that
\begin{align*} \begin{split} \label{riemdef}
\left[ \nabla_a,  \nabla_b \right] v^c\,  &= R_{ab}{}^c{}_d v^d=R_{ab}{}^\hash \, v^c, \\
\left[ \nabla_a,  \nabla_b \right] \hh\omega_c &= R_{abc}{}^d \omega_d=R_{ab}{}^\hash\,  \omega_c \, ,
\end{split}
\end{align*}
where 
$v$ is any tangent vector field  and $\omega$  any 1-form. Using the
metric to raise and lower indices we have, for example,
$R_{abcd}=g_{ce}R_{ab}{}^e{}_d$ and $|v|^2=v_a v^a=g_{ab} v^a v^b$. The Riemann tensor can be decomposed~as
\begin{equation*} \label{riemdec}
R_{abcd} =: W_{abcd} + 2 \left(g_{c[a}P_{b]d} + g_{d[b}P_{a]c}\right),
\end{equation*}
where the {\em Weyl tensor}  $\, W_{abcd}$ is 
completely trace-free 
 and $P_{ab}$ is the {\em Schouten tensor}. It follows,  when $\operatorname{dim} M=:n\geq3$, that
\begin{equation*}
P_{ab} :=  \frac{1}{n-2} \left( R_{ab} - \frac{R}{2 \left( n-1 \right)} g_{ab}\right)
\, ,
\end{equation*}
where $R_{bc}:= R_{ab}{}^a{}_c$ is the {\em Ricci tensor} and its metric trace $R:=g^{ab}R_{ab}$ is the {\em scalar curvature}. We will use $J$ to denote the metric trace of Schouten, {\it i.e.} $J := g^{ab} P_{ab}$. 
The identity endomorphism of ${\mathcal E}^a$ is denoted $\delta^a{}_b$ so that $g^{ac}g_{cb}=\delta^a{}_b$. The metric  trace-free part of a symmetric group of indices is denoted by a subscript $\circ$ so that, for example, $X_{(abc)_\circ}=X_{(abc)}-\frac{3}{d+2} g_{(ab} X_{c)de}g^{de}$.
Importantly, the trace and trace-free parts of the Schouten tensor obey Bianchi-type identities
\begin{align}
\nabla^a P_{ab} &= \nabla_b J\, ,\nonumber\\[1mm]
\nabla_{c} P_{ab} &= \nabla_c P_{(ab)_\circ}+ \hh\frac1{d-1}\hh \hh g_{ab} \nabla^d P_{(dc)_\circ}
\, .
\label{divPo}
\end{align}
The {\em Cotton} and {\em Bach} tensors are respectively given by 
\begin{align}
C_{abc} & := 2 \nabla_{[a}P_{b]c}\, , 
\label{Cotton}
\\
B_{ab}\,  & :=  \nabla^c C_{cab} + P^{cd} W_{cadb}\, .
\label{bachy}
\end{align}

We everywhere assume that metrics $g_{ab}$ are Riemannian, but many of our results extend easily to pseudo-Riemannian structures.
Throughout the article, all structures are assumed to be smooth unless otherwise stated, and all manifolds are taken to be connected. The notation ${\mathcal O}(s^\ell)$ stands for $s^\ell {\mathcal X}$ where ${\mathcal X}$ is some (suitable)  smooth quantity. 

\section{Tractor Calculus}\label{Tiramisu}

Here we review conformal tractor calculus---see
\cite{BEG,CGtams,curry-G,GOpet} for details---and extend those methods to handle couplings to general linear connections.
A conformal manifold $(M^n,c)$ is a smooth manifold of
dimension $n\geq 3$ equipped with an equivalence class $c$ of Riemanian
metrics, where $g_{ab}$, $\widehat{g}_{ab} \in c$~means that
$$\widehat{g}_{ab}=\Omega^2 g_{ab}\, ,$$ for some smooth positive function
$\Omega$. 
Oftentime, calculations will be 
performed with respect to some 
 generic $g \in
c$.

\smallskip

There is no
distinguished connection on the tangent bundle of a general conformal manifold. However there is an invariant and canonical tractor connection on the closely related  standard tractor bundle. To describe these ingredients we first need the notion of a conformal density:
 The top exterior power of the tangent bundle $ \Lambda^{n} TM$ is a line bundle. Thus  its square
$\mathcal{V}:=(\Lambda^{n} TM)^{\otimes 2}$  is
canonically oriented and so one can take oriented roots thereof. Given
$w\in \mathbb{R}$ we set
\begin{equation} \label{cdensities}
\ce[w]:=\mathcal{V}^{\frac{w}{2n}} ,
\end{equation}
and refer to this as the {\it bundle of conformal densities}. For any vector
bundle~$VM$, we write $VM[w]:=VM\otimes\ce[w]$, for example
$\ce_{(ab)}[w]=\odot^2T^* M\otimes
\ce[w]$ and $VM[0]=VM$.
 On a fixed Riemannian manifold, the line bundle ${\mathcal V}$ is canonically
trivialized by the square of the volume form.
Since each metric in a conformal class determines a trivialization of~${\mathcal V}$, it follows easily that on a conformal structure  there is a
canonical section $\bg_{ab}\in \Gamma(\ce_{(ab)}[2])$ termed the {\it conformal metric}. This has 
 the property that any  strictly positive section $\tau\in \Gamma(\ce [1])$ determines a metric
$$g_{ab}:=\tau^{-2}\bg_{ab}\in c\, .$$
 For this reason we call such a density  $\tau$  a {\it scale}.
Moreover, the Levi--Civita connection of~$g_{ab}$ preserves
the corresponding scale $\tau$ and therefore $\bg_{ab}$. Thus it makes sense to use the
conformal metric to raise and lower indices, even when we are choosing
a particular metric representative $g_{ab} \in c$ and its Levi--Civita connection
for calculations.  
Since each $g\in c$ trivializes $\cV$, it also follows that a section $\Gamma(\ce [w])$  may be understood as an equivalence class $[g;f ]$ of pairs $g\in c$ and $f\in C^\infty (M)$, under the relation 
$$
(g,f) \sim (\Omega^2 g , \Omega^w f), \qquad 0<\Omega\in C^\infty (M)\, . 
$$
For this reason, given a bundle $VM[w]$, we call the number $w$ its {\it weight}.

From Taylor series of sections of $\ce[1]$ one recovers the jet exact sequence at~2-jets,
\begin{equation}\label{J2}
0\to \ce_{(ab)}[1]\stackrel{\sf i}{\longrightarrow} J^2\ce[1]\to J^1\ce[1]\to 0 \, .
\end{equation}
The bundle $J^2\ce[1]$ and its sequence \eqref{J2} are canonical objects on any smooth manifold. Given a
 conformal structure $c$, we have the orthogonal decomposition of~$\ce_{(ab)}[1]$  into trace-free and trace parts
\begin{equation*}
\ce_{(ab)}[1]= \ce_{(ab)_0}[1]\oplus \bg_{ab}\cdot \ce[-1]\,  .
\end{equation*}
Thus we can canonically quotient $J^2\ce[1]$ by the image of
$\ce_{(ab)_0}[1]$ under ${\sf i}$ as given in~\nn{J2}. The resulting
quotient bundle is denoted $\cT^*M$, or $\ce_A$ in abstract indices,
and called the {\it conformal cotractor bundle}. 
Its dual  is termed the {\it standard tractor bundle}~$\cT M$.

Since the
jet exact sequence at 1-jets of $\ce[1]$ is
$$ 
0\to \ce_{a}[1]\stackrel{\sf i}{\longrightarrow} J^1\ce[1]\to \ce[1]\to 0\, ,
$$ 
we see  that $\cT^*M$ has a composition series (or filtration structure)
\begin{equation*}\label{filt}
\cT^*M=\ce[1]\lpl \ce_a [1] \lpl \ce[-1]\,  .
\end{equation*}
What this notation means is that $\ce[-1]$ is a subbundle of $\cT^*M$,
and the quotient of $\ce_A=\cT^*M$ by $\ce[-1]$ (which is $J^1\ce[1]$) has
$\ce_a [1]$ as a subbundle, and there is a canonical 
injection
$X_A:\ce[-1]\to \ce_A$. 
We also employ an abstract index notation for {\it tensor tractor bundles} $\cT^\Phi M$ and their sections,
obtained by tensoring~$\ct M$ and~$\ct^*M$. 
Here the label $\Phi$  stands for some collection of abstract tractor indices.

\smallskip

Given a choice of metric $g \in c$, the formula
\begin{equation}\label{thomas-D}
  \ce[1]\ni\sigma \mapsto
  \begin{pmatrix}
    \sigma \\
    \nabla_a \sigma \\
    - \frac{1}{n} \left( \Delta + J \right) \sigma
  \end{pmatrix}\, ,
\end{equation}
where $\Delta $ is the Laplacian $\nabla^a\nabla_a={\sf g}^{ab}\nabla_a\nabla_b$, gives a second-order differential operator on $\ce[1]$ which is a linear map $J^2 \ce[1] \to \ce[1] \oplus \ce_a [1] \oplus \ce[-1]$.
 This map factors through $\cT^*M$,  and a choice of metric $g$ determines an isomorphism
\begin{equation}\label{T_isom}
  \cT^*M \stackrel{\cong}{\longrightarrow} {[\cT^*M]}_g = \ce[1] \oplus \ce_a [1] \oplus \ce[-1]\, .
\end{equation}
We will often use this isomorphism to ``split'' 
tractor bundles.  Thus, given $g \in c$, a section $V_A$ of $\ce_A$ may be labeled by a triple
$(\si,\mu_a,\rho)$ of sections of the bundles in the right hand side of the above display. Equivalently we may write
\begin{equation}\label{Vsplit}
  V_A=\si Y_A+\mu_a Z_A{}^a+\rho X_A\, .
\end{equation}
Here  the algebraic splitting operators
$Y_A:\ce[1]\to \ce_A$ and $Z_A{}^a :\ce_a[1]\to \ce_A$ are determined by the
choice $g \in c$ and can be viewed as sections of $\ce_A[-1]$ and~$\ce_A{}^a [-1]$, respectively, while the tractor $X_A$ is a section of $\ce_A[1]$ determined by $(M,c)$ alone.  We call
the three sections $Y_A$, $Z_A{}^a$ and $X_A$ \emph{tractor projectors}. We will also sometimes use respective monikers ``top/north'', ``middle/west'' and ``bottom/south'' for the ``slots'' $\sigma$, $\mu$ and $\rho$ of $V$.

By construction the tractor bundle is conformally invariant, {\it i.e.} determined by~$(M,c)$ and  independent of any choice of $g\in c$. Evidently the splitting \eqref{Vsplit} is not. Considering the transformation of the operator \eqref{thomas-D} determining the splitting upon moving to a conformally related metric $\widehat{g}=\Omega^2 g$,
one sees that the slots of an invariant section of $\cT^*M$ must be related~by
\begin{equation}\label{ttrans}
[\cT^*M]_{\widehat{g}}\ni \left(\begin{array}{c}
\widehat{\sigma}\\
\widehat{\mu}_b\\[1mm]
\widehat{\rho}
\end{array}\right)=
\left(\begin{array}{ccc}
1 & 0 & 0\\
\Upsilon_b & \delta^c_b & 0\\[1mm]
-\frac{1}{2}|\Upsilon|_g^2 & -\Upsilon^c & 1
\end{array}\right)\!\left(\begin{array}{c}
\sigma\\
\mu_c\\
\rho
\end{array}\right)
\sim
\left(\begin{array}{c}
\sigma\\
\mu_b\\[1mm]
\rho
\end{array}\right)\in [\cT^*M]_g\, ,
\end{equation}
where $\Up_a:=\Omega^{-1}\nabla_a \Omega\in \Gamma(\ce_a)$. Conversely the above
transformation of triples is the hallmark of an invariant tractor
section. Equivalent to the last display is the rule for how the
tractor projectors transform
\begin{equation}\label{XYZtrans}
\textstyle
\widehat{X}_A=X_A, \quad  \widehat{Z}_A{}^{b}=Z_A{}^{b}+\Up^bX_A, \quad
\widehat{Y}_A=Y_A-\Up_bZ_A{}^{b}-\frac12 |\Upsilon|^2_{\sf g}\, X_A \, .
\end{equation}

In terms of splittings determined by a metric $g\in c$, the {\it tractor connection} is given by 
 \begin{equation}\label{tr-conn}
 \nabla_a^\cT  \begin{pmatrix}
    \sigma \\
    \mu_b \\
    \rho
  \end{pmatrix} :=\begin{pmatrix}
   \nabla_a \sigma-\mu_a \\
    \nabla_a\mu_b+P_{ab}\si+\bg_{ab}\rho \\
    \nabla_a\rho- P_{ac}\mu^c
  \end{pmatrix} ,
\end{equation}
where on the right hand side  $\nabla$ denotes the Levi--Civita connection of $g$. Using the transformation of components, as in \eqref{ttrans}, and also the conformal transformation of the Schouten tensor (we have here also recorded those of the Cotton and Bach tensors for completeness),
\begin{eqnarray}\label{Ptrans}
P^{\Omega^2g}_{ab}\: &=&P^g_{ab}-\nabla_a\Upsilon_b
+\Upsilon_a\Upsilon_b-\tfrac{1}{2}\hh g_{ab}|\Upsilon|_g^2\, ,\\[1mm]\nonumber
C^{\Omega^2g}_{abc}\:&=&C^g_{abc}+\Upsilon^dW_{dcab}\, ,\\[2mm]\nonumber
\Omega^2 B^{\Omega^2g}_{ab}\!\!&=&B_{ab} + (d-4)\big[\Upsilon^d C_{d(ab)}-\Upsilon^c\Upsilon^dW_{cabd}\big]\, ,
\end{eqnarray}
reveals that the triple on the right hand side of Equation~\nn{tr-conn} transforms as a 1-form
taking values in $\cT^*M$; {\it i.e.} again by \eqref{ttrans} except
twisted now by $\ce_a$. Thus Equation~\eqref{tr-conn} gives the
splitting into slots of a conformally invariant connection~$\nabla^\cT$  acting on sections of the bundle $\cT^*M$. In fact, from Equation~(\ref{tr-conn}), one may show that 
\begin{align} 
\nabla^{\ct,g}_a X^A\, =&\, \, \: Z^A{}_a \,, \nonumber \\
\nabla^{\ct,g}_a Z^A{}_b =& -P_{ab} X^A - \bg_{ab} Y^A\,  ,\label{dXYZ} \\
\nabla^{\ct,g}_a Y^A\,  =& \,\, \, P_a{}^b Z^A{}_b\, . \nonumber
\end{align}

\medskip

There is a deep  relationship between the tractor connection and Einstein metrics: Using~\eqref{Ptrans} and the transformation of the
Levi--Civita connection, it is straightforward to verify that the
equation
\begin{equation}\label{AE}
\nabla_{(a}\nabla_{b)_\circ} \si+ P_{(ab)_\circ}\si =0
\end{equation}
on conformal densities $\si\in \Gamma(\ce[1])$, is conformally
invariant. This   overdetermined system of PDEs---for which solutions in general do
not exist---may be studied by
prolongation. Indeed    parallel
transport with respect to the connection~\eqref{tr-conn} is exactly the closed system that
arising by prolonging~\eqref{AE}. Non-trivial solutions of \eqref{AE} for $\sigma$ are non-vanishing on an open dense
set. Moreover, the corresponding metric $\si^{-2}\bg_{ab}$ is Einstein. For further details see~\cite{BEG,curry-G,Goal} and references therein.

\smallskip

The tractor bundle is also equipped with a conformally invariant bundle metric~$h_{AB} \in \Gamma (\ce_{(AB)} )$ of signature $(n+1,1)$. This {\it tractor metric}  is defined by  the mapping
\begin{equation}
[V_A]_g = \begin{pmatrix}
    \sigma \\
    \mu_a \\
    \rho
  \end{pmatrix} \mapsto \mu_a \mu^a +2  \si \rho =: h \left(V,V \right),
\end{equation}
and the polarization identity. 
It may be decomposed into a sum of projections
$$
h_{AB}=Z_A{}^a Z_B{}^b\bg_{ab}+X_AY_B+Y_AX_B\, .
$$
 The tractor connection~$\nabla^{\cT}$  preserves the tractor metric~$h$,
so it makes sense to use~$h_{AB}$ to raise and lower tractor indices. In particular the weight one tractor~$X^A=h^{AB}X_B$ is termed the {\it canonical tractor}. It  gives the canonical jet projection $\ce_A\to \ce[1]$.
We will employ a~$\cdot$ notation to denote contractions, for example~$h(U,V)=h_{AB}U^A V^B = U_A V^B = U\cdot V$ for any pair of tractors $U,V$.

The curvature~$\kappa_{ab}{}^C{}_D$ of the tractor
connection is determined by
$$
[ \nabla_{a}^{\ct,\gamma}, \nabla^{\ct,\gamma}_{b}] \, V^C = \kappa_{ab}{}^{C}{}_D V^D \, ,$$
where $V^C$ is any section of the standard tractor bundle and~$\gamma$ is any torsion-free affine connection. 
In terms of tractor projectors one has
\begin{equation} \label{trc2}
  \kappa_{ab}{}^C{}_D =W_{ab}{}^c{}_d Z^{C}{}_c Z_D{}^d
+  C_{ab}{}^c Z^{C}{}_c X_{D} 
-  C_{abd} Z_D{}^d X^{C} 
\, .
\end{equation}
The {\it tractor curvature} is a 2-form taking values in sections of the  bundle $\wedge^2\cT M$. The latter  is  termed the {\em adjoint tractor
  bundle}, as it is a vector bundle modeled on the Lie algebra of the
conformal group $SO(n+1,1)$. 

\smallskip

\begin{remark}\label{EgivesYM}
The Yang--Mills condition with respect to some $g\in c$ for the tractor connection of $(M,c)$ is given by 
\begin{equation}\label{tractorYM}
g^{ab} \nabla^{\ct,g}_a \kappa_{bcDE}=
(d-4) C_{dec} Z_D{}^d Z_E{}^e
+ 2 B_{cd} Z_{[D}{}^d X_{E]}=0\, .
\end{equation}
Here we used the splitting determined by $g$ itself.
When $d=4$, the above-displayed condition  is  equivalent 
to  Bach flatness~\cite{GSS}.
When  $d>4$, it suffices that $g$ is  Einstein to satisfy  Condition~\nn{tractorYM}~\cite{Goal}.
\end{remark}

\smallskip

There exists a rather useful, tractor-valued, extension of the connection-coupled gradient:
The \textit{Thomas-$D$ operator}, mapping  $
 \Gamma( \ct^{\Phi} M[w])$ to $\Gamma(\ct^* M \otimes \ct^{\Phi} M[w-1])$, is given in a choice of splitting by
\begin{align}\label{splitT}
 D_A\, T^{\Phi} :=(d+2w-2) \left(w Y_A  T^{\Phi} + Z_A{}^a  \nabla_a  T^{\Phi}\right) - X_A ({\sf g}^{ab} \nabla_a \nabla_b+ wJ) T^{\Phi}\, .
\end{align}
In the above we have introduced the convention that a connection $\nabla$ is coupled according to the type of section it acts upon. For example, the above splitting uses some choice of metric and $\nabla_b T^\Phi= \nabla^\cT_b T^\Phi$ is a tractor-valued 1-form, so ~$\nabla_a$ acts as the corresponding, tractor-coupled  Levi--Civita connection on this. 
The Thomas-$D$ operator is \textit{strongly invariant} meaning that the above formula also gives a map $
 \Gamma( \ct^{\Phi} M[w]\otimes VM)$ to $\Gamma(\ct^* M \otimes \ct^{\Phi} M[w-1]\otimes VM)$ when $\nabla$ is in addition coupled to any connection $A$ on a vector bundle $VM$ and $T^\Phi$ takes values in sections thereof; 
 see for example~\cite{GOpet,Gover-DN}.

\medskip

One may also construct a curvature-like tractor, dubbed the {\it $W$-tractor}, from the commutator of Thomas-$D$ operators~\cite{GOpet}.
Indeed, when~$T^\Phi \in \Gamma(\ct^{\Phi} M[w])$ and~$n\neq 4$, 
$$[D_A, D_B] \hh T^{\Phi}= (n+2w-4)(n+2w-2) W_{AB}{}^{\sharp} T^{\Phi} + 4 X_{[A} W_{B]C}{}^\sharp \circ D^C T^{\Phi}\,,$$
where, in a choice of splitting,
\begin{align*}
W_{ABCD} := Z_A{}^a Z_B{}^b Z_C{}^c Z_D{}^d W_{abcd} &+ (2X_{[A} Z_{B]}{}^a Z_D{}^c Z_C{}^b   - 2 X_{[C} Z_{D]}{}^a Z_A{}^c Z_B{}^b ) C_{cba} \\
&+ \tfrac{4}{n-4} Z_{[A}{}^a X_{B]} Z_{[C}{}^b X_{D]} B_{ab} \,.
\end{align*}

The Thomas-$D$ operator is not a derivation, however it almost obeys a Leibniz property if suitably rescaled. For that we define
a  \textit{hatted Thomas-$D$ operator} given, when $n+2w-2 \neq 0$, by
$$\hd := \tfrac{1}{n+2w-2} \,D : \Gamma(\ct^{\Phi} M[w]\otimes VM) \rightarrow \Gamma(\ct^* M \otimes \ct^{\Phi} M[w-1]\otimes VM)\,.$$
A useful identity is $$X\cdot\hat D \hh T^\Phi = w T^\Phi\, $$ for any weight $w\neq 1-\frac n2$ tractor $T^\Phi$.
We call any standard tractor $I$ given by the  hatted Thomas-$D$ operator acting on a scale~$\tau$, so that
 $$
 I^A:= \hat D^A \tau\, ,
 $$
 a {\it scale tractor}. The top slot of a scale tractor equals the corresponding scale, 
 \begin{equation}\label{scaletop}X\cdot  I = \tau\, .\end{equation}
 Equation~\nn{AE}, specialized to  the case that $\sigma$ is some scale $\tau$,
 is a necessary and sufficient condition for $c$ to contain  an Einstein metric, and moreover 
  is equivalent to the parallel condition 
  $$
 \nabla_a I^A = 0
 $$
 on the corresponding scale tractor; see~\cite{BEG}.
 The above-displayed condition is  equivalent to requiring $D_A I^B=0$.

\bigskip

The hatted Thomas-$D$ operator's failure to obey a Leibniz rule is encoded below.
\begin{lemma} \label{leibniz-failure}
Let $U_1 \in \Gamma(\ct^\Phi M[w_1]\otimes VM)$ and $U_2 \in \Gamma(\ct^{\Phi'} M[w_2]\otimes VM)$, such that $n+2w_1 -2 \neq 0\neq n+2w_2 -2 $ and $n+2w_1 + 2w_2 -2 \neq 0$. Then,
\begin{equation}\label{Iamafailure}
\hat D_A (U_1 U_2) - (\hat D_A U_1) U_2 - U_1 \hat D_A U_2 =-\tfrac{2}{n+2(w_1+w_2)-2} \, X_A\,  (\hat D_B U_1) \hat D^B U_2\,.
\end{equation}
\end{lemma}

\begin{proof}
This identity was first given in~\cite{Taronna}. A detailed proof, in the case that~$VM$ is the trivial bundle,  may be found in~\cite{Forms}. To see that it still holds for 
general bundles~$VM$, we note that the right hand side of the above display appears precisely because  the connection-coupled Laplacian term appearing in the coefficient of the canonical tractor $X$ in~\nn{splitT} does not obey a Leibnitz identity. It is not difficult to establish that twisting the tractor connection with a bundle connection~$A$ leaves unchanged the form of the product rule for the Laplacian~${\sf g}^{ab} \nabla_a \nabla_b$.
\end{proof}

When acting on totally skew tractors---``tractor forms''---it is propitious to further modify the Thomas-$D$ operator in a manner akin to the Laplace--Beltrami improvement of the rough Laplacian  acting on differential forms.
The \textit{slashed} Thomas-$D$ operator is defined by
$$\Dslash_A := D_A - X_A\circ  W^{\sharp \sharp}\, .$$
When $n+2w-2 \neq 0$, we also define
$\hat{\Dslash} := \tfrac{1}{n+2w-2} \, \Dslash$.
Both $ \Dslash$ and $\hat{\Dslash}$  extend trivially to act on  tractor forms taking values in some other bundle.

\bigskip

There exists a quartet of conformally invariant maps $q_{\sf n}$, $q_{\sf e}$,
$q_{\sf w}$ and $q_{\sf s}$ that 
promote differential forms to tractor forms~\cite{BGforms,BGdeRham,Forms,Silhan}. We shall only need the latter two of these, adapted to act on forms taking values in sections of a vector bundle. In particular
$$
q_{\sf s} : \Gamma(\wedge^\ell T^* M[w]\otimes VM) \rightarrow \Gamma(\wedge^{\ell+1} \ct^* M[w-\ell+1]\otimes VM)\, ,
$$
where 
\begin{equation}\label{antipodal}
q_{\sf s}(t) = X_{[A_1} Z_{A_2}{}^{a_1} \cdots Z_{A_{\ell+1}]}{}^{a_\ell} t_{a_1\cdots a_\ell}\, .
\end{equation}
Note that by taking $VM=T^*M$, the operator $q_{\sf s}$ may be applied to sections of~$T^*M[w]\otimes T^*M$ as has been done in Equation~\nn{propto}; a similar maneuver can even be used to define $q_{\sf s}$ on more general covariant tensors. Also observe that the map ~$q_{\sf s}$ defined by Equation~\nn{antipodal} is 
a bundle map, {\it i.e.} it is algebraic in the sense that it does not involve derivatives of the tensor $t$.

\bigskip

When $w\neq 2\ell-n$, we may also define
$$q_{\sf w} : \Gamma(\wedge^\ell T^* M[w]\otimes VM) \rightarrow \Gamma(\wedge^\ell \ct^* M[w-\ell]\otimes VM)\, ,$$  
where
\begin{align}\label{gowestyoungman}
q_{\sf w}(t):=Z_{[A_1}{}^{a_1} \cdots Z_{A_\ell]}{}^{a_\ell} t_{a_1 \cdots a_\ell} - \tfrac{\ell}{n+w-2\ell} \hh X_{[A_1} Z_{A_2}{}^{a_2} \cdots Z_{A_\ell]}{}^{a_\ell} \nabla^{a_1} t_{a_1 \cdots a_\ell}\, .
\end{align}
So long as $w\neq -\ell$, the map $q_{\sf w}$ has a conformally invariant formal adjoint 
$$q^*_{\sf w} : \Gamma(\wedge^\ell \ct^* M[w]\otimes V^*M) \rightarrow \Gamma(\wedge^\ell T^* M[w+\ell]\otimes V^*M)\, ,$$   
where
\begin{align*}
q^*_{\sf w}(T) :=Z_{a_1}{}^{A_1} \cdots Z_{a_\ell}{}^{A_\ell} T_{A_1 \cdots A_\ell} - \tfrac{\ell}{w+\ell} \nabla_{[a_1} \big(X^{A_1} Z_{a_2}{}^{A_2} \cdots Z_{a_\ell]}{}^{A_\ell} T_{A_1 \cdots A_\ell}\big)\,.
\end{align*}
For later use, note that one can  replace $V^*M$  by $VM$ in the above adjoint map and still obtain an invariant operator.
Also, when $\iota_X T=0$, one has $q^*_{\sf w}(T) :=Z_{a_1}{}^{A_1} \cdots Z_{a_\ell}{}^{A_\ell} T_{A_1 \cdots A_\ell}$ and this map is invariant, even when $w=-\ell$. 
Note that
\begin{equation}\label{iota}
\iota_X \circ q_{\sf w} = 0 = \iota_{\cancel D} \circ q_{\sf w}\, ;
\end{equation}
the above conditions can in  fact be used to determine the map $q_{\sf w}$.
Moreover, if $U$ is any standard tractor,
and $\mu$ a 1-form, then
\begin{equation}\label{lastgasp}
q^*_{\sf w}\big(\iota_U q_{\sf s}(\mu)\big)= \mu\hh\hh   X\cdot  U \, .
\end{equation}

 When $F$ is the curvature of a connection $A$ on  $VM$ and $n\neq 4$, we term  
$$
q_{\sf w}(F)=:{\mathcal F}\in \Gamma(\wedge^2 \ct^*M[-2]\otimes \operatorname{End}\!VM)
$$
the corresponding {\it curvature tractor}~${\mathcal F}$.  In terms of projectors
\begin{equation}\label{nameme}\mathcal{F}_{AB} = Z_A{}^a Z_B{}^b F_{ab} + \tfrac{2}{n-4} Z_{[A}{}^a X_{B]} \nabla^c F_{ca}\,.\end{equation}
The curvature tractor obeys a Bianchi-like identity.

\begin{lemma}\label{tractor-bianchi}
Let $F$ be the curvature of a  connection $A$ on $VM$ and suppose $3\leq n=\operatorname{dim} M\neq 4,6$. 
Then the curvature tractor satisfies
\begin{equation}
\label{Bianchiid}
\hat{\Dslash}_{[A} \mathcal{F}_{BC]} = 0\,.
\end{equation}
\end{lemma}
\begin{proof}
According to the definition of $\hat{\Dslash}$, we have that
$$
\hat{\Dslash}_{[A} \mathcal{F}_{BC]} = \hd_{[A} \mathcal{F}_{BC]} + \tfrac{2}{n-6} X_{[A} W_{B}{}^{DE}{}_{C]} \mathcal{F}_{DE}\,.$$
On sections of $\wedge^2 \ct^* M[-2] \otimes \operatorname{End}\!VM$, the hatted Thomas-$D$ operator $\hat D$  acts, in the splitting given by some a choice of metric representative $g \in c$, according to
$\hd_A = -2 Y_A + Z_A{}^a \nabla_a - \tfrac{1}{n-6} X_A (\Delta - 2 J)$.
Thus, using Equation~\nn{dXYZ} together with the formula for the curvature  tractor  in Equation~\nn{nameme},
we find
\begin{multline*} 
\hat{\Dslash}_{[A} \mathcal{F}_{BC]} = Z_{[A}{}^a Z_B{}^b Z_{C]}{}^c \nabla_{[a} F_{bc]} \\
 + \tfrac{1}{n-6} X_{[A} Z_{B}{}^b Z_{C]}{}^c \big(-\Delta F_{bc} + 2 \nabla_{b} j_{c} - 2(n-4) P_{b}{}^a F_{ca} + 2 J F_{bc} + 2 W_{b}{}^{de}{}_{c} F_{de} \big)\,.
\end{multline*}
The first term on the right hand side vanishes by dint of the  Bianchi identity $\nabla_{[a} F_{bc]} = 0$
which also implies that
$\Delta F_{bc} = -\nabla^a (\nabla_c F_{ab} + \nabla_b F_{ca})$.
Applying the identity 
$[\nabla_a, \nabla_b] F_{cd} = R_{ab}{}^{\sharp} F_{cd} + [F_{ab},F_{cd}]$, and recalling that $j_b := \nabla^a F_{ab}$, 
%
we thus have that
\begin{align*} 
\Delta F_{bc} = 2 \nabla_{[b} j_{c]} - 2(n-4) P_{[b}{}^a F_{c]a} + 2 J F_{bc} + 2 W_{b}{}^{de}{}_{c} F_{de}\,.
\end{align*}
\end{proof}

\section{Yang--Mills Holography}\label{Gnocchi}

Let $X^{d}$ be a compact manifold with boundary $\partial X = M^n$, where $d=n+1$.
A weight $w=1$ density $\bm \sigma \in \ce X[1]$
is called {\it defining} if any representative $\bm s\in  \bm \sigma$ (determined by some $g\in \cc$) is a defining function for the embedding  $M\hookrightarrow X$, meaning $\ext \bm s $ is nowhere vanishing along $M$ and the zero locus of $\bm s$ is $M$ itself. Indeed the unit conormal~$\hat n^g$ 
 to $M$, with respect to $g$,  is given by $(\ext \bm s)/|\ext {\bm s}|_g$. This gives a canonical decomposition 
 $$TX|_M=TM\oplus NM\, ,$$ 
 where $NM$ is the normal bundle. 
 This decomposition is independent of the particular choice of $g\in \cc$ as can be seen from the conformal transformation property of the unit conormal  $\hat n^{\Omega^2 g} =  \Omega\hh \hat n^g$. Indeed, this determines a {\it conformal unit conormal} $\hat n_a\in \Gamma(N^*M[1])$ as well as a {\it conformal
unit normal} vector $\hat n^a = \bm g^{ab} \hat n_b\in \Gamma(NM[-1])$. In the above display we have  used the isomorphism between sections of $TX|_M$ perpendicular to~$\hat n$ with those of~$TM$ to equate these spaces and will even use the same abstract indices for sections of those bundles. For clarity, we will denote sections of bundles over $X$ by a bold notation
  to avoid confusion with those on~$M$. Also, $\bm g$  denotes the conformal metric of $\cc$. 
 
\smallskip

A complete metric $g_+$ 
on the interior $X_+$ is said to be conformally compact if~$\bm \sigma^2 g_+$
smoothly extends to the conformal metric $\bm g$ for a conformal class of  metrics~$\bm c$ on~$X$ whenever $\bm \sigma$ is a defining density.
A conformally compact metric $g_+$   induces a  conformal class of metrics~$c$ on~$M$. 
Slightly abusing notation, we also term a standard tractor $\bm I$ given by
$$\bm I:=\hat{  D} \bm \sigma \in \ct X$$ 
 a {\it scale tractor}. The point here is that $\bm I$ is a scale tractor on the interior $X_+$, but it also extends smoothly to $X$. 
 The parallel condition
 \begin{equation}\label{parallel}
 \nabla_a \bm I^A =0\, ,
\end{equation}
on  the scale tractor of a defining density $\bm \sigma$ 
is then equivalent to the Poincar\'e--Einstein for $g_+=\bm \sigma^{-2}\bm g$~\cite{Goal}. By a constant rescaling of~$g_+$  one may further achieve 
 $$
 \bm I^2 = 1\, .
 $$
 Thus $\bm{ I}|_M$ gives a tractor analog 
 of the unit normal vector and a canonical decomposition $$\ct X|_M = \ct M \oplus {\mathcal N}M\, .$$ 
Indeed,  in a choice of splitting where $\bm I^A  = \bm \sigma \bm Y^A + \bm Z^A{}_a \bm n^a + \bm\rho \bm X^A$, the density-valued vector $\bm n^a=\nabla^a \bm \sigma$ is an extension of the conformal unit normal~$\hat n$ to $X$. 
Also, we will denote by~$\top$ the restriction of tractors on $X$  to $\ct X|_M$ composed with the projector   onto the first summand of the above direct sum.
 Once again we have equated bundles and sections,  and recycled the corresponding abstract indices.  
 Further details are given in
 ~\cite{AGo}.
 Also observe that $\bm I\cdot \bm X 
|_M=0$ so that, for any extension~$\bm T^{\rm ext}$ of a covariant tensor~$T$ along $M$, we have
$\iota_{\bm I} q_{\sf s}(\bm T^{\rm ext})|_M=0$ 
and in turn
 \begin{equation}\label{southissouth}
 q_{\sf s}(\bm T^{\rm ext})|_M =  q_{\sf s}(T)\, .
 \end{equation}  
   Note that when extending covariant tensors on $M$ to $X$ we  first use the orthogonal decomposition $T^*X|_M\cong T^*M \oplus NM$ to trivially extend from  $T^*M$ to $T^*X|_M$.

 \smallskip

On any conformally compact manifold $(X,\cc,\bm \sigma)$ for which the scalar curvature of the conformally compact metric $g_+$ is everywhere strictly negative, 
there is a very useful triple of operators $\{x,h,y\}$ defined acting on tractors (of any weight and tensor type) that obey the
Lie algebra $\mathfrak{sl}(2)$~\cite{GW},
\begin{equation}\label{sl2}
[h,x]=2x\, ,\quad [x,y]=h\, ,\quad [y,h]=2y\, .
\end{equation}
The operator $x$ is simply multiplication by the defining density $\bm \sigma$, while $h$ multiplies tractors of definite weight $w$ by $d+2w$. The operator $y$ is given by
$$
y:=-\frac1{\bm I^2} \bm I\cdot \cancel{D}\, .
$$
Note that the operator $\bm I\cdot \cancel{D}$
is an example of a so-called  {\it Laplace--Robin}
operator  because it may be viewed as a degenerate Laplace--type operator, see~\cite{Gover-DN,GW} and references therein for further details.
Strong invariance of the Thomas-$D$ operator implies that the above $\mathfrak{sl}(2)$ algebra also holds when defined  acting on tractors taking values in sections of a vector bundle $VX$,  given a connection~$\bm A$.

For later use, let us record the following  $\mathfrak{sl}(2)$ 
enveloping algebra relations,
\begin{align}
[x,y^\ell]\, &=y^{\ell-1} \ell (h-\ell+1)\, ,\label{stripoff}\\
[x^\ell,y^\ell] &=(-1)^{\ell+1} \ell!\,  h(h+1)(h+2)\cdots (h+\ell-1) + xF
\, ,\label{factorial}
\end{align}
valid for any strictly  positive integer $\ell$ and where $F$ is some element of the enveloping algebra ${\mathcal U}(\mathfrak{sl}(2))$.
Also, note the identity 
\begin{equation}\label{sigpow}
\hat { D}_A \bm \sigma^\ell = \ell \bm I_A \bm \sigma^{\ell-1} 
-\frac{\ell(\ell-1)}{d+2\ell-2}
 \bm I^2 \bm X_A \bm \sigma^{\ell-2}\, .
\end{equation}
This result was first derived in~\cite{Will1} but also follows from Lemma~\ref{leibniz-failure}.

Let us now suppose that the connection $\bm A$ is  Yang--Mills on  $X_+$ with respect to some conformally compact metric $g_+$ and  
$\bm{\mathcal F}\in \Gamma(\wedge^2 \mathcal{T}X[-2] \otimes \operatorname{End}\!VX)$ is the corresponding curvature tractor.
Then, when $d\neq 4$, a short computation shows that 
$
(d-4)q_{\sf w}^*( \bm I^A \bm {\mathcal F}_{AB})=\bm \sigma \nabla^a 
{\bm  F}_{ab}
 -(d-4) \bm n^a
 \bm F_{ab}$.
As shown in~\cite{GLWZ1}, the right hand side of this equation equals the smooth extension of $\bm \sigma^{-1}$ times the Yang--Mills current $ j[g_+,\bm A]$.  
Thus when $d\neq 4$, 
 the equation
\begin{equation}\label{IF}
\bm I^A{\bm {\mathcal F}}_{AB}=0\, ,
\end{equation}
implies
 the Yang--Mills condition on $\bm A$ over $X_+$ with respect to $g_+$. In fact, on $X^d$ for $d>4$, 
  Equation~\nn{IF} 
is equivalent to imposing  $$\bm \sigma  \nabla^a 
{\bm  F}_{ab}
 -(d-4) \bm n^a
 \bm F_{ab}=0\, .$$

\begin{lemma}
Let $\bm A$ be a Yang--Mills connection on a conformally compact manifold $(X,g_+)$ with corresponding scale tractors and curvature tractors $\bm I$ and $\bm{\mathcal F}$. Then
\begin{equation}\label{harmonic}
 \bm I\cdot \cancel{ D}\, \bm{\mathcal F}=0\, .
\end{equation}
\end{lemma}

\begin{proof} This is a direct consequence of the Bianchi identity of Lemma~\ref{tractor-bianchi} since
$$
\bm I^C  \cancel{ D}_C\bm{\mathcal F}_{AB}
= -2 \bm I^C  \cancel{ D}_{[A}\bm{\mathcal F}_{B]C}
= -2   \cancel{ D}_{[A}(\bm I^C\bm{\mathcal F}_{B]C})=0\, .
$$
The  second to last equality used
 the modified Leibniz rule~\nn{Iamafailure}
 as well as that
 the
parallel condition~\nn{parallel} implies that  $\cancel{ D}_A,\bm I^C=0$.
\end{proof}

We are here not interested in constructing exact solutions to the Poincar\'e--Einstein or Yang--Mills conditions on $X_+$. Rather, we will only need the case where $X = [0,\varepsilon)\times M$ is some collar neighborhood of $M$ where $0<\varepsilon\in {\mathbb R}$. In that case, it is well known~\cite{FG,GW} that the data of conformal class of metrics $c$ on~$M$ determines an asymptotic solution for $g_+$  to the Poincar\'e--Einstein condition to the order given by
 \begin{equation}\label{parallelasym}  
\nabla_a \bm I^A ={\mathcal O}(\bm \sigma^{d-3})
\Rightarrow D_A \bm I^B = {\mathcal O}(\bm \sigma^{d-4})
\, .
\end{equation}
We shall call the data of a collar neighborhood $X=M\times [0,\varepsilon)$
equipped with a conformal class of metrics $\cc$ and defining density $\bm \sigma$ for $M$ that obeys Equation~\nn{parallelasym} and $\bm I^2|_M=1$ an {\it asymptotically Poincar\'e--Einstein} structure $(X,\cc,\bm\sigma)$. 
Note that the Einstein metric on the interior $X_+$ is then given in terms of the conformal metric $\bm g$ and the defining density $\bm \sigma$ by
$$
g_+=\frac{\bm g}{\bm \sigma^2}\, .
$$ 
%
%
%
Moroever, as is evident from the prolongation discussion below Equation~\nn{AE}, Equation~\nn{parallelasym}  
 is equivalent to the statement
 \begin{equation}\label{FGobst}
 P^{g_+}_{(ab)_\circ}= \bm \sigma^{d-3} \bm {\mathcal B}^{\rm ext}_{ab}
\end{equation}
 for some tensor $\bm {\mathcal B}^{\rm ext}\in \Gamma(\odot^2 T^*X_+[3-d])$. In fact, the tensor $\bm  {\mathcal B}^{\rm ext}$ is an  extension to~$X_+$ of the  {\it Fefferman--Graham obstruction  tensor} $  {\mathcal B}\in \Gamma(\odot^2 T^*M[3-d])$. 
 The latter is  trace-free and non-vanishing 
 for all even dimensions $n=d-1\geq 4$.  
 The {\it Fefferman--Graham flatness} condition~$ {\mathcal B}=0$ is a conformally invariant, higher derivative, weakening of the Einstein condition~\cite{FG}.

\smallskip
We can also compute the asymptotic behaviors of the Cotton and Bach tensors.

\begin{lemma}\label{CBasym}
Let $(X^d,\cc,\bm \sigma)$ be an asymptotically Einstein structure where $d\geq 5$. 
Then 
\begin{eqnarray}\label{Casym}
C^{g_+}_{abc}
&=&\!2 (d-2)\hh \bm \sigma^{d-4} \bm n_{[a} {\bm {\mathcal B}}^{\rm ext}_{b]c}+{\mathcal O}(\bm \sigma^{d-3})\, ,
\\
\label{Cnasym}
C^{g_+}_{\bm nbc}
&=&\: \hh (d-2)\hh \bm \sigma^{d-4}\hh {\bm {\mathcal B}}^{\rm ext}_{bc}+{\mathcal O}(\bm \sigma^{d-3})\, ,
\end{eqnarray}
and
\begin{equation}\label{Bgp}
B^{g_+}_{ab}=(d-2)\bm \sigma^{d-3}{\bm {\mathcal B}}^{\rm ext}_{ab} + {\mathcal O}(\bm \sigma^{d-2})\, ,
\end{equation}
where $\bm n=\nabla^{ g} \bm \sigma$ for some $ g\in \cc$.
\end{lemma}

\begin{proof}
We rely on Equation~\nn{FGobst}. Upon noting for   any 1-form field~$\nu$ on~$X_+$ that $$\nabla^{g_+}_a  \nu_b =\nabla^{ g}_a \nu_b 
+\bm \sigma^{-1} \big(\bm n_a \nu_b  + \bm n_b \nu_a -{\bm g}_{ab} \bm n^c \nu_c)\, ,$$
it in turn follows from Equation~\nn{FGobst} that
$$
\nabla_c^{g_+} P^{g_+}_{(ab)_\circ}= \bm \sigma^{d-4} \big[
(d-1) \bm n_c {\mathcal B}^{\rm ext}_{ab}
+
2 \bm n_{(a} {\mathcal B}^{\rm ext}_{b)c}
\big]
+{\mathcal O}(\bm \sigma^{d-3})\, .
$$
Employing
Equations~\nn{divPo} and~\nn{Cotton}  it is then not difficult to
establish Equation~\nn{Casym}. 
Equation~\nn{Cnasym} follows directly from Equation~\nn{Casym}
upon recalling that $\bm n$ extends the unit conormal to $X$ 
and that $\bm n^a    \bm {\mathcal B}^{\rm ext}_{ab}\big|_M=0$.

Along similar lines, Equations~\nn{bachy} and~\nn{FGobst} together give
$$
B_{bc}^{g_+}=g_+^{ad} \nabla_d^{ g} C_{a(bc)}^{g_+}
-(d-5) \bm \sigma C^{g_+}_{\bm n(bc)}+{\mathcal O}(\bm \sigma^{d-2})\, .
$$
Applying  Equation~\nn{Casym} and~\nn{Cnasym} to this completes the proof.
\end{proof}

In mimicry of the above discussion of asymptotically Poincar\'e--Einstein structures, the additional data of a connection $A$ over $M$ determines an asymptotic solution~$\bm A$  to the Yang--Mills condition on $X$ with respect to $g_+$, subject to 
\begin{equation}\label{aaYM}
j[g_+,\bm A]= 
\bm \sigma^{d-3} \bm k=
{\mathcal O}(\bm \sigma^{d-3})
\, ;
\end{equation}
see~\cite{GLWZ1}. Connections $\bm A$ subject to Equation~\nn{aaYM} for some $$\bm k\in \Gamma(T^*X[3-d]\otimes \operatorname{End}\!VX)$$
 are called {\it asymptotically Yang--Mills}. Observe that an asymptotically Yang--Mills connection $\bm A$ determines a corresponding {\it obstruction current} $\bm k$; we shall use this below. In fact, when restricted to the boundary, the obstruction current  gives the conformal Yang--Mills current and so, in turn,  the
  conformal Yang--Mills conditions.
 \begin{theorem}[from \cite{GLWZ1}]\label{holok}
 Let $d-1\in 2{\mathbb Z}_{\geq 2} $ and   $(X^d,\cc,\bm\sigma)$
be asymptotically Poincar\'e--Einstein.
If~$\bm A$ is an asymptotically Yang--Mills connection on $VX$, then 
$$
k[c,A]=\bm k|_M\, ,
$$
where the
connection~$A$ and conformal class~$c$ are those induced along $M$ by~$\cc$ and~$\bm A$, respectively. 
 \end{theorem}

Applying Theorem~\ref{holok}  to a connection~$\bm A$ that is the tractor connection determined by an asymptotically Poincar\'e--Einstein structure 
shows that the Fefferman--Graham obstruction tensor also obstructs the Yang--Mills condition for the tractor connection of $(X,\cc)$.
\begin{theorem}\label{YMFG}
 Let $d-1\in 2{\mathbb Z}_{\geq 2}$ and   $(X^{d},\cc,\bm\sigma)$
 be asymptotically Poincar\'e--Einstein. Also let $\bm A^{\ct}$ be the tractor connection determined by $(X,\cc)$ and let $c$ and $A^\ct$  be the corresponding conformal class and tractor connection induced along~$M$. Then the Yang--Mills current of $\bm A^{\ct}$ with respect to $g_+$ obeys
\begin{equation}\label{needanumber}
j[g_+,\bm A^\ct]=2(d-2)(d-3) \bm \sigma^{d-3} {q}_{\sf s} ({\bm {\mathcal B}}^{\rm ext})
+{\mathcal O}(\bm \sigma^{d-2})\, , 
\end{equation}
where $\bm {\mathcal B}^{\rm ext}$ is some smooth extension of the Fefferman--Graham obstruction tensor~${\mathcal B}$ of $(M,c)$ to $X$.
 \end{theorem}

\begin{proof}We start by considering the Yang--Mills current with respect to $g_+$ for the tractor connection by computing 
$$
g_+^{bc} \nabla^{\ct,g_+}_b \kappa_{ca}{}^B{}_C=
2 q_{\sf s} (B^{g_+})+
(d-4) \bm \sigma
\big[
  C^{g_+}_{\bm n ba}
(\bm Z^{Bb} \bm X_C-\bm Z_C{}^{b} \bm X^B)
+ \bm \sigma C^{g_+}_{bca}\bm Z^{Bb}\bm Z_C{}^c
\big]
\, .
$$
Here the splitting on the right hand side is with respect to some ${ g}\in \cc$. 
Also note that the terms on the right hand side in square brackets define a section of $T^*X\otimes \ct X[-1]$.
By virtue of Lemma~\ref{CBasym}, it follows that 
the coefficient of $\bm Z^{Bb} \bm X_C-\bm Z_C{}^{b} \bm X^B$ in the above-displayed formula is
$$
B_{ab}^{g_+}+(d-4) \bm \sigma C_{nba}^{g_+}=(d-2)(d-3) \bm \sigma^{d-3} {\bm {\mathcal B}}^{\rm ext}_{ab} + {\mathcal O}(\bm \sigma^{d-2})\, .
$$
It then follows from Equation~\nn{Casym} of Lemma~\ref{CBasym} that  the  leading asymptotics of the Yang--Mills current of $\bm A^{\ct}$ with respect to $g_+$ are given by Equation~\nn{needanumber}.
 \end{proof}
 
\noindent
We may now demonstrate equivalence of the Fefferman--Graham obstruction tensor and the conformal Yang--Mills current of the tractor connection.
 
 \begin{proof}[Proof of Theorem~\ref{CYM=FG}] 
 The proof of
Theorem~\ref{YMFG}
shows that the tractor connection~${\bm A}^\ct$  of an  asymptotically Poincar\'e--Einstein manifold $(X^{d},\cc,\bm\sigma)$ is  asymptotically Yang--Mills.
Thus, comparing Equation~\nn{aaYM} with Equation~\nn{needanumber} of Theorem~\ref{YMFG}  we have
$$
\bm k = 2(d-2)(d-3) {q}_{\sf s} ({\bm {\mathcal B}}^{\rm ext})
+{\mathcal O}(\bm \sigma)\ .
$$

We would now like to   apply
 Theorem~\ref{holok}
 to establish Equation~\nn{propto}.
Towards this end 
observe that
the scale tractor $\bm I$  of $\bm \sigma$ is  parallel along the boundary $M=\partial X$ with respect to the tractor connection of~$\bm c$, 
$$
 \nabla_a \bm I^A|_M=0\, ;
$$
see Equation~\nn{parallelasym}.
Moreover in  the decomposition  $\ct X|_M = \ct M \oplus {\mathcal N}M$,   the normal tractor bundle ${\mathcal N}M$ of  the embedding $M\hookrightarrow (X,\cc)$ has fiber spanned by~$\bm I|_M$, which is a parallel tractor restricted to the boundary. 
Hence the connection~$\nabla^\ct$ determines a connection on  $\ct M$ by restriction.
In fact
this restriction of the tractor connection $\nabla^\ct$ of~$\bm c$ to $\ct M$ is equivalent to  the tractor connection
of the boundary conformal class~$c$, so we shall identify these; see~\cite{Goal}.
 Thus the tractor connection~$\nabla^\ct$ on $\ct X|_M$ trivially extends the tractor connection of $(M,c)$.
Using the above identification and Equation~\nn{southissouth}
it now follows that 
$$k[c,A^\ct]=2(d-2)(d-3) q_{\sf s}({\mathcal B})\, .$$ 
\end{proof}

\begin{remark}
In even ambient dimensions $d$, the  obstructions to smoothly solving the Poincar\'e--Einstein condition and the Yang--Mills condition on Poincar\'e--Einstein manifolds both vanish, respectively see~\cite{FG} and~\cite{GLWZ1}. In~\cite{GLWZ1} it is further shown that the conformal Yang--Mills current $k$ is divergence free, 
$$
\nabla^a k_a = 0\, ;
$$
this is a necessary condition for $k$ to be a functional gradient.
Using Equations~\nn{dXYZ}, it is easy to show that  
vanishing divergence of $q_{\sf s}({\mathcal B})$ is equivalent to vanishing of both the trace and divergence   of the Fefferman--Graham obstruction tensor,
$$
{\mathcal B}_a{}^a=0=\nabla^a {\mathcal B}_{ab}\, ;
$$
 necessary conditions for ${\mathcal B}$ to be a functional gradient of a conformally invariant energy functional. That ${\mathcal B}$ obeys the above display dates back to~\cite{FG}.
These features are examples of the many close parallels between the Einstein and Yang--Mills systems; see~\cite{GLWZ1} and~\cite{GLWZ2}.
\end{remark}

\bigskip

Returning to our study of general connections on bundles over $X$, 
we next need an asymptotic analog of Equation~\nn{IF}.

\begin{theorem}\label{aYM}
Let   $(X^d,\cc,\bm\sigma)$
be asymptotically Poincar\'e--Einstein where $d>4$.
If~$\bm A$ is an asymptotically Yang--Mills connection on $VX$, then 
$$
\bm I^A{\bm {\mathcal F}}_{AB} =\bm \sigma^{d-4} \bm {\mathcal K}_B+\bm X_B {\mathcal O}(\bm \sigma^{d-4})\, ,
$$
where $\bm {\mathcal K}={q}_{\sf w}(\bm k)\in \Gamma(\ct^* X[2-d]\otimes \operatorname{End}\!VX)$.
\end{theorem}
\begin{proof}
Equation~\nn{aaYM} implies that asymptotically Yang--Mills connections obey
$$
\bm \sigma
\nabla^a 
{\bm  F}_{ab}
 -(d-4) \bm n^a
 \bm F_{ab}=\bm \sigma^{d-4}\, \bm k_b\, .
$$
We now apply $q_{\sf w}$ to both sides of the above equation.  Using Equation~\nn{gowestyoungman}, it is not difficult to verify that the left hand side becomes $\bm I^A{\bm {\mathcal F}}_{AB}$. In~\cite{GLWZ1}, it is shown that $\hat n^a \bm k_a|_M=0$. Using this, Equation~\nn{gowestyoungman} can be employed to show that ${q}_{\sf w}(\bm \sigma^{d-4} \bm k) = \bm \sigma^{d-4} {q}_{\sf w}(\bm k) + \bm X {\mathcal O}(\bm \sigma^{d-4})$.
\end{proof}

We will also need the analog of the generalized harmonic Condition~\nn{harmonic} for asymptotically Yang--Mills connections.

\begin{theorem}\label{howdysam}
Let $\bm A$ be an asymptotically Yang--Mills connection. Then 
\begin{equation}\label{kinfo}\bm I\cdot  \cancel{D} \hh\hh \bm {\mathcal F}_{AB}
=-2(d-4)(d-5)
\bm \sigma^{d-6}\bm X_{[A} \bm {\mathcal K}_{B]}
+{\mathcal O}(\bm \sigma^{d-5})\, ,
\end{equation}
where 
 $\bm{\mathcal  K}$ is as given in Theorem~\ref{aYM}.
\end{theorem}

\begin{proof}
The result follows from  a direct computation based on the Bianchi identity~\nn{Bianchiid}, 
the asymptotic parallel condition~\nn{parallelasym}, Theorem~\ref{aYM} and 
  Equation~\nn{sigpow}: 
  \begin{align*}
\bm I\cdot  \cancel{D} \bm {\mathcal F}_{AB}&=
-2 \bm I^C \cancel{D}_{[A}\bm F_{B]C} =2 \cancel{D}_{[A}\big( \bm \sigma^{d-4} \bm {\mathcal K}_{B]}+\bm X_{B]}{\mathcal O}(\bm \sigma^{d-4}) \big)
+{\mathcal O}(\bm \sigma^{d-4})\\
&=2(d-6)\bm {\mathcal K}_{[B} \hat{D}_{A]}\bm \sigma^{d-4}
-4\bm X_{[A}
( \cancel{\hat{D}}^C \bm {\mathcal K}_{B]})
\hat {D}_C
\bm \sigma^{d-4} 
+{\mathcal O}(\bm \sigma^{d-5})
\\&
=-2(d-4)(d-5)
\bm \sigma^{d-6}\bm X_{[A} \bm {\mathcal K}_{B]}
+{\mathcal O}(\bm \sigma^{d-5})\, .
\end{align*}
\end{proof}

We next would like to extract information about $\bm k$ along the boundary $M$  from Equation~\nn{kinfo} of Theorem~\ref{howdysam} since, by dint of Theorem~\ref{holok}, this gives the conformal Yang--Mills conditions.
Using Equation~\nn{factorial} one sees that powers of the operator~$\bm I$~$\csdot$~$\cancel{D}$ can be employed to remove powers of~$\bm \sigma$ along $M$ where the operator~$x$ vanishes. On the other hand, from Equation~\nn{stripoff} we see, as observed in~\cite{GW}, that certain powers  commute with the operator $x$ when acting on tractors of weight such that $h-\ell+1$ vanishes. This observation leads to generalizations of the conformally invariant Laplacian powers of~\cite{GJMS}.
Indeed, consider the map 
${\sf ext}$
that maps sections of $\ct^{\Phi} M[w]\otimes VM$ to equivalence classes of sections of  $\ct^{\Phi} X[w]\otimes VX$ with equivalence defined by equality upon restriction. Here we first trivially extend sections of $\ct^\Phi M$ to
sections of $\ct^\Phi X|_M$ before extending these to $X$.
Operators ${\sf O}$ on sections of $\ct^{\Phi} X[w]\otimes VX$ that obey 
$${\sf O} \circ x = x \circ \widetilde {\sf O}\, ,$$ where $\widetilde {\sf O}$ is some operator on
$\Gamma(\ct^{\Phi} X[w-1]\otimes VX)$, are well-defined acting on these equivalence classes. Such an operator is given, in $d$ dimensions,  by  $(\bm I \csdot \cancel{D})^{d-5}$ in the case that $w=-2$ and $\bm I^2 = 1 + {\mathcal O}(\bm \sigma^{d-6})$.
Hence we define the following, connection-coupled, generalization of the GJMS operators
\begin{equation}\label{IDpow}
\cancel{P}_{n-4} :=\top \circ  (\bm I\hh \csdot \hh  \cancel{D})^{d-5} \circ {\sf ext} : \Gamma(\ct^{\Phi} M[-2]\otimes VM) \rightarrow \Gamma(\ct^{\Phi} M[2-n]\otimes VM)\,.
\end{equation}
Note that we have focused on the special case $w=-2$ since that is the weight of the curvature tractor $\bm{\mathcal F}$. However, by virtue of strong invariance of the Thomas-$D$ operator, the construction
 presented here extends immediately to the construction of connection-coupled GJMS operators at more general weights. We can now prove Theorem~\ref{parmesan}.
 
\begin{proof}[Proof of Theorem~\ref{parmesan}]
Suppose that we know that a weight $w=-2$ tractor~$\bm {\mathcal F}$ solves 
$$
\bm I\cdot  \cancel{D}\hh  \bm {\mathcal F}
=\bm \sigma^{d-6} \bm {\mathcal Q}\, ,
$$
for some $\bm {\mathcal Q}$.
It then follows from Equation~\nn{factorial}
that acting with $d-6$ powers of~$\bm I\cdot  \cancel{D}$
on the above display and restricting to $M$ gives
$$
 \big((\bm I\cdot  \cancel{D})^{d-5}  \bm {\mathcal F}\big)\big|_M
=(-1)^{d-6}[(d-6)!]^2\, \bm {\mathcal Q}|_M \, .
$$
Note here we have used that since $\nabla_a \bm I^A ={\mathcal O}(\bm \sigma^{d-3})$, one has $\bm I^2 =1 +  {\mathcal O}(\bm \sigma^{d-1})$,
so that $y$ in Equation~\nn{factorial} can be replaced by $-\bm I \cdot \cancel{D}$.
Applying the above reasoning to Equation~\nn{kinfo} gives
$$
 \big((\bm I\cdot  \cancel{D})^{d-5}  \bm {\mathcal F}\big)\big|_M
=2(-1)^{d-5}(d-4)(d-5)[(d-6)!]^2\,
\bm X_{[A} \bm {\mathcal K}_{B]}
 \big|_M \, .
$$
Recall that  $\bm {\mathcal K}={q}_{\sf w}(\bm k)$, so   Equation~\nn{gowestyoungman} implies that 
$$
\bm X_{[A}\bm {\mathcal K}_{B]}= \bm X_{[A} \bm Z_{B]}{}^a \bm k_a
=q_{\sf s}(\bm k)\, .
$$
The proof of Theorem~\ref{holok} given in~\cite{GLWZ1} shows that  $\bm k|_M = k$. 
Employing Equation~\nn{southissouth} completes the proof.
\end{proof}

We may now prove Theorem~\ref{main}.
\begin{proof}[Proof of Theorem~\ref{main}]
We begin by examining Equation~\nn{P=k}. 
By virtue of  its  construction from powers of the connection coupled Laplace--Robin operator in Equation~\nn{IDpow}, the operator~$\cancel{P}_{n-4}$ has a formula
$$
\cancel{P}_{n-4} {\mathcal F}_{AB}= \sum_{\ell=0}^{d-5} P^{a_1\cdots a_{\ell}} \nabla_{a_1} \cdots \nabla_{a_\ell}  {\mathcal F}_{AB}\, ,
$$
where each covariant derivative $\nabla$ above is tractor-, Levi--Civita- and $V$-bundle connection-coupled. Here $P^{a_1\cdots a_{\ell}}$ is some $V$-bundle endomorphism-valued tensor. 
Now suppose that $\tau$ is the scale determining the Einstein metric $g_{\rm E}$ and $I_\tau$ the corresponding scale tractor. Then we have
$$
\nabla_a I_\tau^A=0\, ,
$$
from which it  follows that 
$$
I_\tau^A\cancel{P}_{n-4} {\mathcal F}_{AB}
= \cancel{P}_{n-4} I_\tau^A  {\mathcal F}_{AB} = 0\, .
$$
The latter equality holds because $A_{\rm YM}$ is Yang--Mills with respect to $g_{\rm E}$ and ${\mathcal F}$ is the corresponding curvature tractor.
The proof is completed by employing Equation~\nn{lastgasp} together with Equation~\nn{scaletop}
 and the last statement of Theorem~\ref{parmesan}  to note that 
$$
q_{\sf w}^*\big( \iota_{I_\tau}
q_{\sf s}(k)
\big) = \tau k\, .
$$
\end{proof}

It is well known that every four dimensional  Einstein metric is Bach flat. The analogous result for the Fefferman--Graham obstruction tensor is also known~\cite{FG}.
One might wonder whether this is a special case of Theorems~\ref{main} and~\ref{YMFG} applied to the tractor connection. 
Indeed this is the case.
\begin{corollary}[of Theorems~\ref{main} and~\ref{YMFG}]
Let $n \in 2 \mathbb{Z}_{\geq 3}$ and let $(M^n,c)$ be a conformal manifold. Then if there exists an Einstein metric $g_E \in c$, the Fefferman--Graham obstruction tensor~${\mathcal B}$ of $(M^n,c)$ vanishes.
\end{corollary}
\begin{proof}
As explained in Remark~\ref{EgivesYM},
the existence of an Einstein metric~$g_{\rm E}\in c$, implies that the tractor connection 
is Yang--Mills with respect to $g_{\rm E}$.
So by Theorem~\ref{main}, the tractor connection obeys the conformal Yang--Mills condition on~$(M,c)$. 
Theorem~\ref{CYM=FG} shows that
 the conformal Yang--Mills condition  implies that the Fefferman--Graham obstruction tensor vanishes. 
 \end{proof}

%

\subsection{Examples}\label{examples}

Our results apply to connections on any vector bundle $VM$, so in particular  when $VM$ is  a line bundle. Our first example addresses our claim in the introduction  that not every conformal Yang--Mills connection is itself Yang--Mills.

\begin{example}\label{Euclid}
Let $V{\mathbb R}^6$ be a  line bundle over 6-dimensional Euclidean space~$({\mathbb R}^6,\delta)$.
Then we may describe a connection by a 1-form $A$ and its  curvature by~$F=\ext A$. In those terms, the conformally invariant Yang--Mills condition~\nn{k} becomes
$$
\Delta j=0\, ,
$$ 
where the current is given by the 1-form  $j=\ext^* F$.
Let $\lambda$ be any parallel vector field and $E$ be the standard Euler vector field defined such that ${\mathcal L}_E\hh  \delta = 2 \delta$. Then consider the 1-form
$$
A=  E^\flat(\lambda) \, E^\flat\, ,
$$
where $E^\flat = \delta(E,\pdot)$. A simple computation shows
$$
j= -5 \lambda^\flat\, ,
$$
which is clearly in the kernel of $\Delta$. Hence we have found a connection that is not Yang--Mills but is conformal Yang--Mills.
\end{example}

\medskip

Finally we give a Yang--Mills connection that is not conformal Yang--Mills. By dint of Theorem~\ref{main}, for that we must consider Riemannian manifolds that are not conformal to Einstein.

\begin{example}
We wish to consider a metric $g$ admitting a non-vanishing, parallel 2-form $F$ so that  $F$ is both closed and co-closed and therefore obeys the source-free Maxwell's equations.
In that case the conformal Yang--Mills Condition~\nn{k} on a 6-manifold reduces to 
\begin{equation}\label{k6}
F^{ab}(C_{abc}-g_{ca}\nabla_b J)=0\, .
\end{equation}
We now search for  a  geometry such that
the left hand side of the above display does not vanish. For that, on ${\mathbb R}^6\ni(u,v,x,y,z,w)$,  we consider a metric ansatz
\begin{equation}\label{g6}
g=\delta + 2\sqrt{1-f(u+v)^2} \hh du dv\, ,
\end{equation}
and a 2-form
$$
F=f(u+v) \hh du \wedge dv\, .
$$
Rudimentary computations show that
for generic functions $f$, the metric $g$ is not conformal to an Einstein metric and nor is the left hand side of Equation~\nn{k6} zero.
However 
 the 2-form $F$ is covariantly constant. 
 We  note that it is even possible to choose $f$ such that Equation~\nn{g6} globally defines a metric on ${\mathbb R}^6$.
\end{example}

%
%
%

\section*{Acknowledgements}
A.R.G. and A.W. acknowledge support from the Royal Society of New Zealand via Marsden Grant 24-UOA-005.
A.W. was also supported by Simons Foundation Collaboration Grant for Mathematicians ID 686131.
S.B. would like to acknowledge the generous financial support of Matt Gimlin as well as the U.C. Davis QMAP center  for hospitality.


\begin{thebibliography}{10}

\bibitem{BPST}
A.~A. Belavin, A.~M. Polyakov, A.~S. Schwartz and Yu.~S. Tyupkin,  Pseudoparticle solutions of the Yang-Mills equations, {\it Phys. Lett. B.} 59 85--87, 1975.


\bibitem{BEG}
T.~N. Bailey, M.~G. Eastwood, and A.~R. Gover.
\newblock Thomas's structure bundle for conformal, projective and related
  structures,
\newblock {\it Rocky Mountain J. Math.}, 24(4):1191--1217, 1994.

\bibitem{BGforms}
T. Branson, A.~R. Gover, Conformally invariant operators, differential forms, cohomology and a generalisation of Q-curvature, {\it Comm. Partial Differential Equations} 30, 1611--1669, 2005.

\bibitem{BGdeRham}
T.~P. Branson and A.~R. Gover. Pontrjagin forms and invariant objects related to the $Q$-curvature, {\it Comm.  Cont. Math,} 9: 335--358, 2007.



\bibitem{CGtams}
A.~\v{C}ap, and A.~R. Gover.
\newblock Tractor calculi for parabolic geometries,
\newblock {\em Trans. Am. Math. Soc.}, 354(4):1511--1548, 2001.

\bibitem{curry-G}
S.~N. Curry and A.~R. Gover.
\newblock An Introduction to Conformal Geometry and Tractor Calculus, with a view to Applications in General Relativity,
\newblock {\em Asymptotic Analysis in General Relativity},
\newblock London Mathematical Society Lecture Note Series. Cambridge University Press; 86--170, 2018.



\bibitem{Don1}
S. Donaldson, An Application of Gauge Theory to Four Dimensional Topology, {\it J. of Diff. Geom.}, 18 279--315, 1983. 

\bibitem{Don2}
S. Donaldson, P.B. Kronheimer, The geometry of four-manifolds, {\it Oxford Mathematical Monographs}, Clarendon Press, Oxford University Press, New York, 1990. 

\bibitem{FG}
C. Fefferman and C.~R. Graham,
\newblock Conformal invariants,
\newblock Asterisque, Number {N}um\'{e}ro {H}ors {S}\'{e}rie, 95--116, 1985,
\newblock The mathematical heritage of \'{E}lie Cartan (Lyon, 1984).



\bibitem{GLWZ1}
A.~R.~Gover, E.~Latini, A.~Waldron and Y.~Zhang,
\newblock Conformally compact and higher conformal Yang-Mills equations,
\newblock arXiv:2311.11458 [math.DG], 2024.



\bibitem{GLWZ2}
A.~R.~Gover, E.~Latini, A.~Waldron and Y.~Zhang,
\newblock Renormalized Yang-Mills Energy on Poincar{\'e}-Einstein Manifolds,
\newblock arXiv:2409.06995 [math.DG], 2024.



\bibitem{GPS}
A.~R. Gover, L.~J. Peterson, and C. Sleigh.
\newblock  A conformally invariant Yang–Mills type energy and equation on 6-manifolds,
\newblock {\it Commun. Contemp. Math}, 26(2):2250078, 2024.

\bibitem{GSS}
A.~R. Gover, P. Somberg, V. Sou\v{c}ek, Yang--Mills detour complexes and conformal geometry, {\it Comm. Math. Phys.} 278, 307--327, 2008.



\bibitem{GOpet}
A.~R. Gover and L.~J. Peterson.
\newblock Conformally invariant powers of the {L}aplacian, {$Q$}-curvature, and
  tractor calculus.
\newblock {\em Comm. Math. Phys.}, 235(2):339--378, 2003.

\bibitem{Goal}
A.~R. Gover,
\newblock Almost {E}instein and {P}oincar\'{e}-{E}instein manifolds in
  {R}iemannian signature,
\newblock {\em J. Geom. Phys.}, 60(2):182--204, 2010.

\bibitem{Gover-DN}
A.~R. Gover,
\newblock Conformal Dirichlet--Neumann maps and Poincar\'e--Einstein manifolds,
\newblock {\em SIGMA}, 3:100-120, 2007.

\bibitem{AGo}
A.~R. Gover, Almost conformally Einstein manifolds and obstructions,  {\em Differential geometry and its applications}, 247–260, Matfyzpress, Prague, 2005.

\bibitem{Forms}
A.~R. Gover, E. Latini, and A. Waldron,
\newblock Poincar\'{e}--{E}instein holography for forms via conformal geometry
  in the bulk,
\newblock {\em Mem. Amer. Math. Soc.}, 235(1106):VI--95, 2015.


\bibitem{GW}
A.~R. Gover and A. Waldron,
\newblock Boundary calculus for conformally compact manifolds,
\newblock {\em Indiana Univ. Math. J.}, 63(1):119--163, 2014.


\bibitem{Will1}
A.~R. Gover and A.~Waldron, Conformal hypersurface geometry via a boundary Loewner--Nirenberg--Yamabe problem,
 {\it Commun. Anal. Geom.}, 29(4):779–836, 2021.
 
\bibitem{GJMS}
C.~R. Graham, R.  Jenne, L. Mason, 
              G.~A.~J. Sparling, 
      {Conformally invariant powers of the 
      {L}aplacian. {I}.
              {E}xistence},
   {\it J. London Math. Soc. (2)},
     {46}: 557--565, 1992.

 
 \bibitem{Kaku}
 M. Kaku, P.K. Townsend and P. van Nieuwenhuizen,
 Properties of conformal supergravity,
{\it Phys. Rev.} D 17, 3179, 1978. 
 
  \bibitem{Taronna}
E.~Joung, M.~Taronna and A.~Waldron,
A Calculus for Higher Spin Interactions,
{\it JHEP} {07}, 186, 2013.


\bibitem{Penrose}
R.~Penrose,
\newblock Structure of space-time,
\newblock \textit{Battelle Rencontres}, 121--235, 1968.
     


\bibitem{Silhan}
J. \v{S}ilhan, Invariant differential operators in conformal geometry, Ph.D. thesis, The University of Auckland, 2006.


\bibitem{YandM}
C.~N. Yang, R. Mills, Conservation of isotopic spin and isotopic gauge invariance, {\it Phys. Rev.} 96, 191--195, 1954. 




%
%
%
%
%










%
%
















  







%


























 


 





%
%




































\end{thebibliography}
\end{document}